\documentclass[letterpaper, 10pt, conference]{ieeeconf}

\IEEEoverridecommandlockouts
\usepackage{hyperref}
\usepackage{amsmath}
\usepackage{amssymb}
\usepackage{xcolor}
\usepackage{cite}
\usepackage{graphicx}

\usepackage{algorithm}
\usepackage{algpseudocode}

\newtheorem{theorem}{Theorem}[section]
\newtheorem{lemma}[theorem]{Lemma}
\newtheorem{corollary}[theorem]{Corollary}

\newtheorem{definition}[theorem]{Definition}
\newtheorem{remark}{Remark}[section]

\newtheorem{problem}{Problem}

\DeclareMathOperator{\klsep}{\|}
\newcommand{\dkl}[2]{\subscr{D}{KL}(#1 \klsep #2)}

\newcommand{\vnomt}{u_{\textup{nom},t}}
\newcommand{\vnom}{\subscr{u}{nom}}

\newcommand{\Obs}{\mathcal{O}}
\renewcommand{\Pr}{\mathbb{P}}
\newcommand{\U}{\mathcal{U}}
\newcommand{\Vor}{\mathcal{V}}

\newcommand{\rhohat}{\widehat{\rho}_{\mathbf{x}_N,r}}
\newcommand{\rhohati}{\rhohat^{\, i}}

\newcommand{\real}{\ensuremath{\mathbb{R}}}

\newcommand{\realnonnegative}{\ensuremath{\mathbb{R}}_{\ge 0}}

\newcommand{\until}[1]{\{1,\dots, #1\}}
\newcommand{\subscr}[2]{#1_{\textup{#2}}}

\newcommand{\dist}{\operatorname{dist}}

\newcommand{\neighbors}{\mathcal{N}}

\renewcommand{\epsilon}{\varepsilon}

\newcommand{\argmin}{\ensuremath{\operatorname*{arg\,min}}}

\usepackage{psfrag}

\newcommand{\diff}{\mathrm{d}}
\newcommand{\supp}{\operatorname{supp}}

\newcommand{\bulletsym}{\hbox{$\bullet$}}
\newcommand{\bulletend}{\relax\ifmmode\else\unskip\hfill\fi\bulletsym}

\newcommand{\longthmtitle}[1]{\mbox{}{\textit{(#1):}}}

\newcommand{\oprocendsymbol}{\hbox{$\bullet$}} 
\newcommand{\oprocend}{\relax\ifmmode\else\unskip\hfill\fi\oprocendsymbol}

\title{\LARGE \bf
  Banach Control Barrier Functions for Large-Scale Swarm Control
}

\author{Xuting Gao$^*$, Guillem Pascual$^*$, Scott Brown, and
  Sonia Mart{\'\i}nez\thanks{$^*$Equal contribution.  X. Gao,
    G. Pascual, S. Brown, and S. Mart{\'\i}nez are with the Department of
    Mechanical and Aerospace Engineering at the University of
    California, San Diego, CA, USA. (e-mail:
    \{gpascualastivill, xug003, sab007, soniamd\}@ucsd.edu.)}}

\begin{document}

\maketitle

\begin{abstract}
  This paper studies the safe control of very large multi-agent
  systems via a generalized framework that employs so-called Banach
  Control Barrier Functions (B-CBFs). Modeling a large swarm as
  probability distribution over a spatial domain, we show how B-CBFs
  can be used to appropriately capture a variety of macroscopic
  constraints that can integrate with large-scale swarm
  objectives. Leveraging this framework, we define stable and filtered
  gradient flows for large swarms, paying special attention to optimal
  transport algorithms. Further, we show how to derive agent-level,
  microscopical algorithms that are consistent with macroscopic
  counterparts in the large-scale limit. We then identify conditions
  for which a group of agents can compute a distributed solution that
  only requires local information from other agents within a
  communication range. Finally, we showcase the theoretical results
  over swarm systems in the simulations section.
\end{abstract}

\section{Introduction}
Control Barrier Functions (CBFs) have emerged as a powerful and
versatile tool for guaranteeing the safety of dynamical
systems~\cite{ADA-SC-ME-GN-KS-PT:19}. By encoding safety requirements
as inequality constraints on the state space, CBFs ensure forward
invariance of the safe set while remaining compatible with
optimization-based control designs~\cite{ADA-XX-JWG-PT:17}.  This
framework has naturally extended to the domain of multi-agent
systems~\cite{PG-JC-ME:17-cdc, DP-DMS-PGV:13}. 
%
%
In very large swarms, where the number of agents can reach thousands
or more, it becomes both advantageous and necessary to model the
system as a continuum or density to correctly capture
macroscopic performance. Yet, the resulting algorithms need to be
implemented at the agent level, which requires establishing a
meaningful relationship between macroscopic and microscopic algorithms
while accounting for the agents' sensing and communication
limitations. In this work, we investigate this question by means of
macroscopic barrier functions defined over Banach spaces.

Compared to finite-dimensional counterparts, the application of CBFs
in infinite-dimensional settings is largely open. While CBFs have been
implemented for certain functional cases, e.g.~time-delay
systems~\cite{AKK-TGM-ADA-GO:23}, their application to very large
swarms remains limited~\cite{SWF-LN:25}. While the latter studies the
benefits of macroscopic collision avoidance, other aspects of
macroscopic density shaping and how to deal with limitations at each
agent are left unaddressed.

The control of swarm densities is closely tied to transport
formulations, where the goal is to steer an initial distribution
towards a target one. This perspective naturally connects to Optimal
Transport (OT)~\cite{CV:08,JDB-YB:00} as a main algorithmic tool. As
stochastic systems can be modeled as deterministic PDEs evolving in
density space, these approaches also find application in stochastic
control systems~\cite{YC-TTG-MP:15}.  These results highlight that
(deterministic) density steering and optimal transport are both
relevant for deterministic large-swarm and stochastic systems control.

However, directly incorporating safety constraints into OT or other
density-shaping methods remains highly challenging. The addition of
operational requirements such as collision avoidance or density caps
turns the problem into a large, infinite-dimensional constrained
program that is generally intractable. Existing work on density
shaping addresses only special cases---for example, flux
bounds~\cite{AS-AD-TTG:25}, mass constraints~\cite{PG-EK:22}, or
covariance control in Gaussian settings~\cite{JP-PT:24}. A general
method for enforcing broad classes of safety constraints during
density evolution therefore remains an outstanding challenge.

Some methods for density control are used in a centralized manner,
leveraging macroscopic solutions to PDE or OT
formulations~\cite{GF-SF-TAW:14,PF-MK:11,YC:23} that are transmitted
at initial time to all agents in the swarm. While this offers
significant computational efficiency, it is not robust to single-point
failures and has limited adaptability. Instead, in decentralized
approaches agents calculate the solution on their own by locally
interacting with others. While existing solutions are less general,
they emphasize scalable communications, are more robust and
adaptive. Existing methods include distributed OT
algorithms~\cite{SB-SJC-FYH:14,VK-SM:23-auto} or density feedback
schemes based on Markov chain
dynamics~\cite{SB-SJC-FH:17,SB-KE-SB:21}.  While these methods provide
effective steering mechanisms, results on safety guarantees for the
swarm as a whole remains relatively scarce.  Many efforts exploit
problem-specific structures, such as duality-based linear inequality
constraints on Markov chain dynamics~\cite{ND-BA:14}, or bio-inspired
potential field approaches for collision avoidance in UAV
swarms~\cite{JL-YF-HC-ZW-ZW-MZ:22}. A systematic framework for
integrating CBF-style safety guarantees with transport-based swarm
control and distributed implementation is still lacking.

This work presents a unified framework for the safe control of very
large-scale swarms. Our main contributions are the following.
(1) We introduce the use of Banach Control Barrier Functions (B-CBFs),
which generalizes CBFs to those with a co-domain in a
partially-ordered Banach space. These allow for the unified treatment
of both stochastic and spatially-dependent constraints in a way that
fits naturally with large models of swarms as probability
distributions.
  %
  %
(2) We leverage this framework to define gradient and filtered flows
on probability distributions, particularly focusing on constrained
optimal transport problems governed by macroscopic B-CBFs.  (3) We
derive a microscopic density steering formulation, showing it is
consistent with the macroscopic large-scale limit. Finally, we
identify conditions under which a swarm can compute its solution in a
distributed manner when subject to limited-range interactions.

\textit{Notation:} Through out the paper, we let $\real$ denote the
set of real numbers and $\real^n$ the $n-$dimensional Euclidean
space. The set of nonnegative real numbers is defined with
$\realnonnegative$. We let
$\nabla = ( \frac{\partial}{\partial x_1}, \ldots,
\frac{\partial}{\partial x_n} )$ be the gradient operator defined on
$\real^n$, while $\partial_t := \frac{\partial}{\partial t}$ is a
shorthand for the time derivative operator. The support of a function
is denoted by $\supp(\rho) = \overline{\{ x \mid \rho(x) > 0 \}}$. The
Euclidean distance between two points $x,y$ is denoted as
$\dist(x,y)$. The ball of radius $R$, centered at the point $x$ is
given by $B(x,R)$, its volume given by $|B_R|$. The note $\oplus$
defines the Minkowski sum, such that
$A\oplus B := \{x + y\mid x\in A, \, y \in B\}$. Given a set $S$, we
define the indicator function $\mathbf{1}_S(x) = 1$ if $x \in S$ and
$\mathbf{1}_S(x) = 0$ otherwise. The probability of an event $A$ is
denoted by $\Pr[A]$.

\section{preliminaries}\label{sec:pre}
Before stating our core problems, we introduce essential
concepts from functional analysis and optimal
transport~\cite{HB:10,CV:08}.
\subsection{Measure Spaces and Functionals}

In what follows, $\mathcal{P}(\real^n)$ denotes the set of absolutely
continuous measures over $\real^n$.
Any $\mu \in \mathcal{P}(\real^n)$ has an associated density function
$\rho : \real^n \to \realnonnegative$ such that
$\diff \mu = \rho \,\diff x$, where $\diff x$ (resp.~$\diff \mu$)
denotes integration w.r.t.~the Lebesgue measure (resp.~$\mu$) on
$\real^n$. Recall that the set of absolutely continuous measures
$\mathcal{P}(\real^n)$ can be embedded into $L^1(\real^n)$, the
Lebesgue-integrable functions over $\real^n$, which is a complete
normed vector (Banach) space.
%

Next we briefly review the Fr\'{e}chet derivative of a function
between Banach spaces.  We refer the reader to \cite{AS:16b} for more
background on calculus in normed spaces.

\begin{definition}[Fr\'{e}chet derivative]
  Let $Z$ and $W$ be real Banach spaces. A function $F : Z \to W$ is
  differentiable at $z_0$ if there exists a bounded linear operator
  $\frac{\delta F}{\delta x}: Z \to W$ such that
  \begin{gather*}
    \lim_{\|z - z_0\| \downarrow 0}
    \frac{\| F(z) - F(z_0) - \frac{\delta F}{\delta z}(z - z_0) \|}%
    {\|z - z_0\|} = 0.
  \end{gather*}
  The operator $\frac{\delta F}{\delta z}$ is known as the Fr\'{e}chet
  derivative of $F$.
\end{definition}

\subsection{Optimal Transport} \label{subsec:OT}
Here we cite two important definitions from Optimal Transport.
We refer to these tools when computing a
nominal swarm velocity in Section~\ref{sec:OT-guided}.

\begin{definition}[Transport Map]
  Let $\rho_{0}$ and $\rho_{*}$ be probability density functions on
  $\real^{n}$, with associated probability measures
  $\mu_{0}, \mu_{*} \in \mathcal{P}(\real^{n})$. A measurable map
  $T:\real^{n}\to\real^{n}$ is a transport map from $\mu_0$
  to $\mu_*$ if it pushes $\mu_{0}$ forward to $\mu_{*}$, denoted by
  $T_{\#}\mu_{0}=\mu_{*}$. In other words, for any measurable set
  $B\subset\real^{n}$,
  \begin{equation*}
    \mu_{T}(B)=\mu_{0}\bigl(T^{-1}(B)\bigr).
  \end{equation*}
\end{definition}

With a transport map defined, the optimal map is one that minimizes
total transport cost.

\begin{definition}[Monge Problem]
  \label{def:Monge}
  Given two probability measures
  $\mu_{0}, \mu_{T} \in \mathcal{P}(\real^{n})$ and a cost
  function $c:\real^{n}\times\real^{n}\to\real_{\ge0}$,
  the Monge problem is defined as
\begin{equation}\label{eq:OT}
  \begin{aligned}
  \inf_{T} \quad & \int_{\real^{n}} c\bigl(x,T(x)\bigr) \, \diff\mu_{0}(x) \\
  \text{subject to} \quad & T_{\#}\mu_{0}=\mu_{T},
  \end{aligned}
\end{equation}
and the solution to the problem is the optimal transport plan $T_*$.
A notable special case arises when $c(x,y) = \Vert x - y \Vert_p^p$,
in which the optimal cost
defines the $p$-Wasserstein distance. We denote by $W_p(\mu,\zeta)$ the $p$-Wasserstein
 distance between two probability measures $\mu,\zeta$. When these measures admit
 densities $\rho,\eta$, we write $W_p(\rho,\eta)$ with a slight abuse of notation.
The $p$-Wasserstein distance is well-defined if the measure $\mu, \zeta$ has
finite bounded $p$ moments.
\end{definition}


\section{Problem Formulation}\label{sec:problem_formulation}
To develop algorithms for the coordination of large-scale swarms we
will (1) identify macroscopic dynamics and constraints that can be
included in meaningful large-scale coordination objectives, (2) obtain
provably-correct macroscopic safe algorithms, and (3) derive
decentralized and consistent algorithms for a finite group of agents.

Thus, we begin by presenting a continuum model of large swarms.  This
foundation will be used to formulate the three problems that are the
main focus of this work. In particular, we address (2) in
Sections~\ref{sec:computing_controllers} and~\ref{sec:OT-guided} while
the focus of (3) is left for Section~\ref{sec:distributed}.
We adopt a continuum approach to model a large swarm, where the
macroscopic state of the system is described by a time-dependent
density function $\rho : \real \times \real^n \to \realnonnegative$ on
a spatial domain $\real^n$. Note that, in this way, a sample of
$\rho$ corresponds to an agent state.  This macroscopic model is
derived from the microscopic behavior of the individual agents, whose
trajectories $x(t)$ are governed by the fully actuated single
integrator dynamics
\begin{equation} \label{eq:micro-ODE}
  \dot{x} = u(t, x).
\end{equation}
The velocity field $u : \real \times \real^n \to \real$ is to be
designed to achieve the desired macroscopic objectives, such as
steering the swarm to perform a coverage task or maintain safety.

At the macroscopic level, conservation of agents gives rise to the
Liouville equation, a first-order PDE of the form
\begin{equation} \label{eq:macro-LVE}
  \partial_t \rho(t,x) + \nabla \cdot \rho(t,x) u(t,x) = 0,
\end{equation}
subject to a zero-flux boundary conditions at infinity. A rigorous
derivation of this model can be found in e.g.~\cite{NB-CD:11}. We
assume existence and uniqueness of solutions to this equation, which can
be ensured under certain regularity conditions; see~\cite{LA:08}.
%
%
For notational convenience, we denote
$\rho_t(x) := \rho(t,x) \in \real$, for
$t \in \real,\, x \in \real^n$, and, with a slight abuse of notation,
$\rho_t\equiv \rho_t \, \diff x \in \mathcal{P}(\real^n)$, as a
density (measure) on $\real^n$ at time $t$. Similarly, we will denote
the velocity field as $u_t(x) =u(t, x) \in \real^n$ or simply as
$u_t \in \U = L^2(\real^n)$.


In swarm robotics, safety requirements typically encompass canonical
concerns such as obstacle avoidance and conflict avoidance. 
In the following, we provide two main examples for these, establishing
their  connection with microscopic constraints. Let
  $\rho \in \mathcal{P}(\real^n)$ be a density describing a very large
  swarm, and $(x_i,x_j)$ independent samples from $\rho$.

\begin{lemma}[Obstacle Avoidance]\label{lemma:Obstacle}
  Let $\Obs \subseteq \mathbb{R}^n$ be a closed obstacle and let
  $\Obs^b = \Obs \oplus B_{\subscr{d}{min}}$ be a region to avoid,
  where $\subscr{d}{min}$ is a safety margin or agent-size bound. Consider a functional $H:\mathcal{P}(\real^n) \to \real$
  defined as
  \begin{equation}
    \label{eq:avoid}
    H(\rho) = \varepsilon - \int_{\Obs^b}\rho (x) \diff x.
  \end{equation}
  Then, the \textit{macroscopic formulation of obstacle avoidance},
  $H(\rho) \ge 0$, is equivalent to the \textit{microscopic
    formulation of obstacle avoidance},
  $\mathbb{P}[\dist(x_i, \Obs) \ge \subscr{d}{min}] \ge 1 -
  \varepsilon$.
\end{lemma}

\begin{proof}
  Note that
  $\Pr[\dist(x_i,\Obs) \ge \subscr{d}{min}] = \Pr[\dist(x_i,\Obs^b) >
  0] = \int_{\real^n \setminus \Obs^b} \rho(x) \,\diff x = 1 -
  \int_{\Obs^b} \rho(x) \,\diff x$.
\end{proof}
%

Another safety requirement is congestion control, where the goal is to
maintain a minimum distance $d$ between any two agents
of the swarm. This can be achieved by the following equivalent
formulations.
%
%

\begin{lemma}[Conflict Avoidance]\label{lemma:CG}
  Consider the functional $H^d : \mathcal{P}(\real^n) \to \real$, with
  \begin{gather}
    \label{congestion}
    H^d(\rho) =
    \int_{\real^n \times \real^n}
    \mathbf{1}_{ \{\operatorname{dist}(x,y) \le d \} } \,
    \rho(x)\rho(y)\,\diff x \diff y.
  \end{gather}
  Then, for some $\epsilon > 0$, the \textit{macroscopic formulation
    of conflict avoidance}, $H^d(\rho)\le \epsilon$, is equivalent to
  the \textit{microscopic formulation of conflict avoidance},
  $\Pr[\dist(x_i,x_j)\geq d] \ge 1-\epsilon$. Furthermore, a
  sufficient condition for either to hold is the \textit{congestion
    condition}
  $\rho \le \subscr{\rho}{max}\leq \frac{\epsilon}{|B_{d}|}$.
\end{lemma}

\begin{proof}
  The functional in~\eqref{congestion} is equivalent to the
  probability of collision,
  $H^d(\rho) = 1 - \Pr^{\rho\otimes\rho}[\dist(x,y)\geq d]$, where
  $\rho \otimes \rho $ refers to the product measure of
  $\rho \,\diff x$.

  To prove the second statement, we use $\rho \le \subscr{\rho}{max}$
  to bound
  \begin{align*}
    H^d(\rho)
    &=\int_{\real^n} \left(\int_{B(y,d)}
      \rho\diff x\right)\rho\diff y \leq \subscr{\rho}{max} |B_d|,
  \end{align*}
  which then shows that
  $\subscr{\rho}{max}\leq \frac{\epsilon}{|B_{d}|} \implies
  H^d(\rho)\leq \epsilon$.
\end{proof}

Similar reasoning also leads to a method for ensuring agents in the
swarm stay close together.
\begin{corollary}[Swarm Cohesion]
  The \textit{macroscopic formulation of swarm cohesion},
  $H^d(\rho) \ge 1-\epsilon$, is equivalent to the \textit{microscopic
    formulation of swarm cohesion},
  $\Pr[\dist(x_i,x_j)\geq d] \le \epsilon$. Furthermore, a sufficient
  condition for either to hold is
  $\rho \ge \subscr{\rho}{min} \ge \frac{1-\epsilon}{|B_{d}|}$
  .\oprocend
\end{corollary}


Other macroscopic constraints can be formulated in terms of
functions over $\mathcal{P}(\real^n)$. These include Wasserstein,
and divergence-based costs. We discuss their possible use and meaning
in the sections that follow.





Now, we state the central problems addressed in this paper:
\begin{problem}[Safe Filtering of a Velocity Field]
\label{prob:safety_filtering}
Given the swarm dynamics \eqref{eq:macro-LVE}, a safe set
$\mathcal{C} \subseteq \mathcal{P}(\real^n)$, and a time-dependent
nominal velocity field $\vnomt \in \U$ find a velocity field
$u_t(x)$ that renders the set $\mathcal{C}$ forward invariant while
minimizing the deviation from $\vnomt$.
\end{problem}

\begin{problem}[Safe Density Steering]
  \label{prob:ot_guided_steering}
  Using the solution to Problem~\ref{prob:safety_filtering}, design a
  controller to move the swarm from an initial density
  $\rho_0 \in \mathcal{C}$ to a target density
  $\rho_* \in \mathcal{C}$, with the trajectory $\rho_t$ remaining in
  $\mathcal{C}$, for all $t \ge 0$.
\end{problem}
  \begin{problem}[Distributed Implementation]
    Develop fully distributed implementations of the solutions to
    Problems~\ref{prob:safety_filtering}
    and~\ref{prob:ot_guided_steering} that converge to the macroscopic
    solutions as the number of agents is increased.
  \end{problem}

\section{Safe Density Steering Via Quadratic Programming with B-CBF}
\label{sec:computing_controllers}

In this section, we address Problem~\ref{prob:safety_filtering} by
introducing a generalization of the control barrier function, which we
call a Banach-CBF (B-CBF). This notion unifies and generalizes the
notions of control barrier function \cite{ADA-XX-JWG-PT:17}, and
mean-field control barrier function \cite{SWF-LN:25}. It extends the
codomain of the barrier function from the reals to any partially
ordered Banach space. By defining the barrier in this way, it is
easier to formulate certain constraints which arise naturally in the
study of swarm dynamics.

For individual agents with dynamics \eqref{eq:micro-ODE}, safety
constraints may model, for example, obstacle avoidance or velocity
limits; see Section~\ref{sec:problem_formulation}. These safe sets can usually
be represented by the sublevel sets of smooth, real-valued functions
over the state space. However, as discussed here, for a more general
PDE system such as \eqref{eq:macro-LVE}, there are other types of
safety constraints that do not fit well into this
paradigm. Interestingly, PDEs such as \eqref{eq:macro-LVE} can be
viewed as \emph{ODEs} over the appropriate Banach space. In the case
of \eqref{eq:macro-LVE}, we can interpret $\rho_t$ as the state of an
ODE over $L^1(\real^n)$.

Let $\Xi$ and $\mathcal{U}$ be real Banach spaces and define a control
system on $\Xi$ with dynamics $f : \Xi \times \mathcal{U} \to \Xi$,
\begin{gather}
  \label{eq:sys-banach}
  \dot{\xi} = f(\xi, u).
\end{gather}
Note that the Liouville equation \eqref{eq:macro-LVE} is a special
case of \eqref{eq:sys-banach}, with $\rho \simeq \xi$,
and the operator $f(\rho, u) = -\nabla \cdot \rho u$.

We now define our notion of safety in this setting. Let $\mathcal{Y}$
be a real Banach space and $\mathcal{K} \subset \mathcal{Y}$ a
cone.\footnote{A cone is a convex set $\mathcal{K}$ such that for any
  $y_1, y_2 \in \mathcal{K}$ and $\lambda \in \real$,
  $y_1 + y_2 \in \mathcal{K}$ and $\lambda y_1 \in
  \mathcal{K}$.}. $\mathcal{K}$ induces a partial order $\succeq$ on
$\mathcal{Y}$, defined by $x \succeq y \iff x - y \in
\mathcal{K}$. Given a function $H : \Xi \to \mathcal{Y}$, we define a
safe set via the partial order as
$\mathcal{C} = \{ \xi \in \Xi : H(\xi) \succeq 0 \}$.
\begin{definition}
  $H$ is called a \emph{Banach CBF} (B-CBF) if for every
  $\xi_0 \in \mathcal{C}$, there exists a time-varying velocity field
  $u(t, \xi)$ such that $H(\xi)$ is differentiable with respect to
  time and
  \begin{gather}
    \label{eq:b-cbf-constr}
    \dot{H}(\xi) \succeq -\alpha(H(\xi)),
  \end{gather}
  along trajectories\footnote{In other words, by $\dot{H}(\xi)$, we
    mean the time derivative of the composite function
    $H \circ \xi : \real \to \mathcal{Y}$. This is well defined as
    long as the composition is differentiable, but we will typically
    require that $H$ (and trivially $\xi$) are differentiable so that
    $\dot{H}$ can be computed using the chain rule.}  of
  \eqref{eq:sys-banach} with initial condition $\xi_0$, where
  $\alpha : \mathcal{Y} \to \mathcal{Y}$ is continuous, strictly
  increasing with respect to $\succeq$, and $\alpha(0) = 0$.
\end{definition}

The following theorem states that the existence of a B-CBF is
sufficient for forward invariance of the corresponding safe set. We
assume that $f$, $H$, and $\alpha$ are sufficiently regular so that
all ODEs' solutions are unique and exist for all time.
\begin{theorem}\label{thm:CBF}
  Let $H : \Xi \to \mathcal{Y}$ be a B-CBF for
  system~\eqref{eq:sys-banach}. Then the safe set $\mathcal{C}$ is
  forward invariant.
\end{theorem}

\begin{proof}
  To show this, we apply the comparison principle, which holds on
  partially ordered Banach spaces.  We let
  $y(0) = H(\xi_0) \in \mathcal{Y}$, then consider the comparison
  system $\dot{y} = -\alpha(y)$. Since $\alpha(0) = 0$ and $\alpha$ is
  monotone increasing, $y(0) \succeq 0$ implies $y(t) \succeq 0$ for
  all $t > 0$. In other words, the solution $y(t)$ evolves
  on~$\mathcal{K}$. Therefore, the comparison principle in
  \cite[Theorem 2]{VL-ARM-RWM:77} holds. Using this theorem,
  $\dot{H}(\xi) \succeq \dot{y}$ implies $H(\xi(t)) \succeq y(t)$ for
  all $t \ge 0$. This proves $H(\xi(t)) \succeq 0$ for all $t \ge 0$,
  rendering $\mathcal{C}$ forward invariant.
\end{proof}

To highlight the flexibility of this approach, we give several
possible functions $H$, along with a description of how they might be
practically used.  The key is to make the B-CBF
constraint~\eqref{eq:b-cbf-constr} explicit in terms of the velocity
field, so that it can be applied in a convex optimization problem.

\subsection{Scalar-Valued Functionals ($\mathcal{Y} = \real$)}
The mean-field CBF \cite{SWF-LN:25} is a special case of a B-CBF where
the codomain is the set of real numbers,
$(\mathcal{H}, \|\cdot\|) = (\real, |\cdot|)$, with the standard
partial order $\succeq$ being $\ge$. The B-CBF constraint becomes the
familiar scalar inequality $\dot{H}(\rho) \ge -\alpha(H(\rho))$. In
most cases, the time derivative is computed using the chain rule.
%
%
\begin{lemma}\longthmtitle{General Integral Functionals}
  %
  %
\label{lem:General-CBF}
Given a measurable subset $S \subset \real^n$, define a functional
\begin{equation}
  \label{eq:integral-cbf}
  H(\rho) := \beta - \int_S h(x, \rho(x)) \diff x,
\end{equation}
for some continuously differentiable kernel
$h : S \times \realnonnegative \to \real$ and constant
$\beta \in \real$. Then, its derivative along trajectories of
\eqref{eq:macro-LVE} at time $t$ is given by
\begin{equation*}
    \dot H = \int_S
     \frac{\partial h}{\partial \rho}(x) \nabla
    \cdot \rho_t(x) u_t(x) \diff x. \tag*{$\bullet$}
\end{equation*}
\end{lemma}
%
%
%
In the following, we provide examples of such functionals and their
derivatives.
\begin{lemma}[Obstacle Avoidance Constraint]
  %
  %
  \label{lemma:chance-constraints}
  Following Lemma~\ref{lem:General-CBF},
  %
  %
  the time derivative of the functional \eqref{eq:avoid} is
  \begin{equation*}
    \dot H(\rho) = \int_{\Obs^b} \nabla \cdot \rho_t(x) u_t(x) \,\diff x .
    \tag*{$\bullet$}
  \end{equation*}
\end{lemma}

\begin{lemma}[Conflict Avoidance]
  The time derivative of the conflict avoidance functional
  \eqref{congestion} is given by
  \begin{gather*}
    \dot{H}^d = -2 \int_{\real^n \times \real^n}
    \mathbf{1}_{ \{\operatorname{dist}(x,y) \le d\} }
    \rho_t(y) \nabla \cdot \rho_t(x) u_t(x) \,\diff x \diff y.
    \tag*{$\bullet$}
  \end{gather*}
\end{lemma}

%
%

  To constrain a density $\rho$ to remain close to (or far from) a
  static reference density $\eta$, a barrier function can be
  constructed using the Kullback-Leibler (KL) divergence as
  \begin{gather*}
    H(\rho) := \beta - \dkl{\rho}{\eta}
    = \beta - \int_{\real^n} \rho(x) \log\left(\frac{\rho(x)}{\eta(x)}\right) \diff x,
  \end{gather*}
  which is well defined if $\supp(\rho) = \supp(\eta)$. The KL
  divergence $\dkl{\rho}{\eta}$ is not a metric but satisfies
  $\dkl{\rho}{\eta} = 0$ iff $\rho = \eta$ a.e. The following can be
  shown.

\begin{lemma}[KL Divergence]\label{lemma:kl}

  %
  %
  Following Lemma~\ref{lem:General-CBF}, if $\rho_t$ and $\eta$
    are continuously differentiable, the time derivative is
  \begin{equation*}
    \dot H = -\int_{\real^n} \left( \frac{\nabla \rho_t}{\rho_t}
      - \frac{\nabla \eta}{\eta} \right) \cdot \rho_t u_t \,\diff x.  \tag*{$\bullet$}
  \end{equation*}
\end{lemma}

Alternatively, one can use $p$-Wasserstein metric, $W_p(\eta,\rho)$ as
defined in Section~\ref{sec:pre}, to evaluate the differences between
$\eta$ and $\rho \in \mathcal{P}(\real^n)$. The following can be
shown.

\begin{lemma}[Wasserstein Distance]\label{lemma:wasserstein}
Let the barrier functional be constructed from the squared
2-Wasserstein distance to a static reference density $\eta$ as
%
%
$H(\rho) := \beta - \frac{1}{2}W_2^2(\rho, \eta)$. Its time derivative
is given by
\begin{equation*}
    \dot{H} = \int_{\real^n} \nabla \phi(x) \cdot \rho_t u_t \,\diff x,
\end{equation*}
where $\phi$ is the so-called Kantorovich potential from the optimal
transport problem between $\rho$ and $\eta$ (cf. \cite[Proposition
7.17]{FS:15}). The vector field $-\nabla \phi(x)$ represents the
direction of the Wasserstein gradient flow, providing the steepest
descent direction to minimize $W_2^2(\rho, \eta)$.
\oprocend
\end{lemma}

\subsection{Spatially-Dependent Functions ($\mathcal{Y} = L^1(\real^n)$)}
A critical class of safety requirements that are unique to swarms
involves constraints that must hold pointwise across the entire
spatial domain, such as a density cap $\rho(x) \le \subscr{\rho}{max}$ for
all $x \in \real^n$. Formulating such a condition using a real-valued
function such as $H(\rho) = \max_x\; (\subscr{\rho}{max}(x) - \rho(x))$ is
problematic, since it is not differentiable. On the other hand, this
constraint is easily modeled with a differentiable B-CBF.

Recall that the codomain of the B-CBF can be a Banach space of
functions, such as $\mathcal{Y} = L^1(\real^n)$. This allows us to map
a density to a spatially-dependent function, thereby enforcing an
infinite number of constraints simultaneously.\footnote{In other
  words, given uncountably many safe sets
  $\mathcal{C}_x = \{\rho \mid \subscr{\rho}{max}(x) - \rho(x) \ge 0 \}$, the
  B-CBF framework enables to reason about the behavior of their
  intersection $\mathcal{C} = \bigcap_{x \in \real^n} \mathcal{C}_x$
  without extra difficulty.}  We endow $L^1(\real^n)$ with the partial
order $\succeq$ defined by
\begin{equation}\label{eq:order}
    \phi_1 \succeq \phi_2 \iff \phi_1(x) \ge \phi_2(x)
    \quad \forall x \in \real^n.
\end{equation}
The B-CBF inequality \eqref{eq:b-cbf-constr} is then interpreted as a
pointwise-in-space constraint on the velocity field.

\begin{lemma}[Pointwise Density Cap]
  \label{le:density-threshold}
  To enforce the constraint $\rho \preceq \subscr{\rho}{max}$ (i.e.,
  $\rho(x) \le \subscr{\dot{\rho}}{max}(x) \ \forall x \in \real^n$) that ensures
  conflict avoidance for a reference profile $\subscr{\rho}{max}$, we define
  $H : L^1(\real^n) \to L^1(\real^n)$ with
  $H(\rho) =  \subscr{\rho}{max} - \rho$.
  The safe set $\mathcal{C}$ contains all densities pointwise upper
  bounded by $\subscr{\rho}{max}$. The time derivative is therefore
  computed pointwise (assuming $ \subscr{\dot{\rho}}{max}= 0$), as
  $\dot{H}(\rho_t) = \subscr{\dot{\rho}}{max} - \dot{\rho_t} = \nabla
  \cdot \rho_t u_t$. \oprocend
\end{lemma}

The B-CBF inequality \eqref{eq:b-cbf-constr} becomes a pointwise
constraint on the flux divergence,
$(\nabla \cdot \rho u)(x) \ge
-\alpha_x(\subscr{\rho}{max}(x)-\rho(x))$ for all $x \in \real^n$,
which ensures the density cap is respected.
  %

Finally, we apply the notion of B-CBF to address the safety filtering
objective in Problem~\ref{prob:safety_filtering}. At each time $t$, we
solve an infinite-dimensional optimization problem,
\begin{equation} \label{eq:safety_filter_general_B-CBF}
\begin{aligned}
  u_t=\argmin_{u'\in \mathcal{U}} & \quad \int_{\real^n} \|u'(x) - \vnomt(x)\|^2 \;\diff x ,\\
  \text{s.t.} & \quad \dot{H}(\rho_t) \succeq -\alpha(H(\rho_t)),
\end{aligned}
\end{equation}
which is a quadratic program (QP) as long as $u_t$ appears linearly in
the CBF constraint. We introduce tractable methods for solving these
infinite-dimensional programs in Remark~\ref{rem:discrete-qp} and
Section~\ref{sec:distributed}.
%

\section{Safe and Stable Density Steering}
\label{sec:OT-guided}


In this section, we present our core methodology to solve
Problem~\ref{prob:ot_guided_steering}. We start by recalling a
general approach to the construction of $\subscr{u}{nom}$ from
gradient flows, which specializes to a case of optimal transport and
constraints as in Section~\ref{subsec:OT}. Then, we establish
conditions under which this approach is exponentially convergent to
a target.

\subsection{Nominal Velocity Field}\label{sec:nominal}

A first approach to obtain a nominal velocity control field that
steers $\rho_0$ to $\rho_*$ can be found by means of a gradient flow
in the following sense.

\begin{lemma}[\hspace{0.001em}{\cite[Theorem 4.2]{VK-SM:18-cdc}}]\label{le:gradient}
  Consider a real-valued, differentiable functional
  $F:\mathcal{P}(\real^n) \to \real$.  Define

  \begin{equation}\label{eq:nom-gradient}
    \vnomt = -\nabla
    \frac{\delta F}{\delta \rho}(\rho_t).
  \end{equation}
  Starting from a given $\rho_0$, consider the dynamic evolution of
  $\rho_0$ under the dynamics~\eqref{eq:macro-LVE}, subject to
  $u_t = \vnomt$ for all $t \ge 0$. Then, $\dot{F} \le 0$. In
  addition, if $F$ is strongly convex with a minimizer at $\rho_*$,
  $\rho_t$ will weakly converge to $\rho_*$. \oprocend
\end{lemma}






%

Consider $F(\rho) = \frac{1}{2} W_2^2(\rho, \rho_*)$, which satisfies
the assumptions of Lemma~\ref{le:gradient}. As noted in
Lemma~\ref{lemma:wasserstein}, the general (Wasserstein) gradient flow
of $F(\rho) = \frac{1}{2} W_2^2(\rho, \rho_*)$ is given
by $\subscr{u}{nom} = -\nabla \phi$, where $\phi$ is the Kantorovich
potential.
Following~\cite{YB:91},
$ \subscr{u}{nom}(t,x) = -\nabla \phi(x) = T^*(x) - x,$
where $T^*$ is the solution to the OT problem~\eqref{eq:OT}.
As an alternative to solving for $\phi$, one can calculate a nominal
OT velocity field by directly finding $T^*$.  We can employ a
receding-horizon calculation to find $T_t^*$ steering $\rho_t$ to
$\rho_*$, and obtain
\begin{equation}\label{eq:nom-ot}
  \subscr{u}{nom}(t,x) = \gamma(t) \left( T_t^*(x) - x \right),
\end{equation}
where $\gamma(t) \ge 0$ is a time-varying gain, or ``speed schedule''.

\begin{lemma}[1D OT Map \cite{FS:15}] \label{lemma:1D-OT} Let
  $\rho_t\, \diff x, \rho_* \, \diff x\in \mathcal{P}(\real)$ be two
  probability density functions with corresponding cumulative density
  functions (CDFs) $F_0$ and $F_*$. The optimal transport map from
  $\rho_t$ to $\rho_*$ is $T^*_t(x) = F_*^{-1} ( F_0(x) )$. \oprocend
\end{lemma}
A numerical method for higher dimensional OT maps can be found in
\cite{GP-MC:17}. We will apply the distributed method
\cite{VK-SM:23-auto} in Section~\ref{sec:distributed}.

\subsection{Safety and Stability via Quadratic Programming}
At any instant $t$, the nominal velocity is projected onto the set of
safe and stable velocities by solving an infinite-dimensional
Quadratic Program (QP). Let $V:\mathcal{P}(\real^n) \to \real$ be a
Lyapunov functional satisfying
\begin{equation*}
  \varphi_1(m(\rho, \rho_*)) \le V(\rho) \le \varphi_2(m(\rho,\rho_*))
\end{equation*}
for some class-$\mathcal{K}$ functions $\varphi_1, \varphi_2$ and some
metric $m$ over $\mathcal{P}(\real^n)$. Note that
$V$ is not necessarily the same as $F$. At each $t \ge 0$, one can
solve for
\begin{subequations}\label{eq:QP-CBF}
\begin{align}
  u_t = \argmin_{u' \in \U} &
  \quad \int_{\real^n} \| u'(x) - \vnomt(x) \|^2 \,\diff x
  \label{eq:QP-objective} \\
  \text{s.t.} & \quad \dot H(\rho_t) \succeq - \alpha_1(H(\rho_t)),
  \label{eq:QP-CBF-constraint} \\
  & \quad \dot V(\rho_t) \le -\alpha_2 (V(\rho_t)),
  \label{eq:QP-CLF-constraint}
\end{align}
\end{subequations}
where $\alpha_1$ is the same as what we defined in B-CBFs and $\alpha_2$
is a scalar valued class-$\mathcal{K}$ function.
The second constraint, a standard Lyapunov condition, is
designed to enforce convergence and stability to the
target.

\begin{theorem}
  \label{thm:continuous_guarantees}
  Consider the swarm \eqref{eq:macro-LVE} with initial condition
  $\rho_0 \in \mathcal{C}$ and target density
  $\rho_* \in \mathcal{C}$. Assume that $V(\rho)$ has a
    globally unique minimizer at $\rho_*$.
    Define a safe set $\mathcal{C}$ by
a B-CBF $H$. Assume that for any $t \ge 0$, the optimization problem
\eqref{eq:QP-CBF} is feasible.
%
  %
  Then, the trajectory of the density $\rho_t$, evolving under $u_t$
  the solution to \eqref{eq:QP-CBF}, satisfies
  \begin{enumerate}
  \item \emph{Safety:} The trajectory remains in the safe set, i.e.,
    $H(\rho_t) \succeq 0$ for all $t \ge 0$.
  \item \emph{Convergence:} The trajectory $\rho_t$ asymptotically
    converges to $\rho_*$ as $t\to \infty$, as characterized by the
    metric~$m$.
  \end{enumerate}
\end{theorem}

\begin{proof}
  The first part, safety, follows directly from
  constraint~\eqref{eq:QP-CBF-constraint} and Theorem~\ref{thm:CBF}.
  The convergence guarantee is derived from the Control Lyapunov
  Functional (CLF) constraint \eqref{eq:QP-CLF-constraint}
  following~\cite[Theorem 3.3.6]{ANM-LH-DL:08}.
  %
\end{proof}
\begin{remark}[Role of $V$]
  Note that if $F$ is not a strictly convex functional of $\rho$
  and/or the B-CBFs do not define a convex set
  $\mathcal{C} \subseteq \mathcal{P}(\mathbb{R}^n)$, the solution
  to~\eqref{eq:QP-objective}-\eqref{eq:QP-CBF-constraint} may lead to
  a local minimizer or saddle
  point. Introducing~\eqref{eq:QP-CLF-constraint} can help reach
  $\rho_*$. However, if $F$ is a strictly convex functional of $\rho$,
  the B-CBFs $H(\rho) \succeq 0$ defines a convex set
  $\mathcal{C} \subseteq \mathcal{P}(\mathbb{R}^n)$, and
  $\rho_* \in \mathrm{Int}(\mathcal{C})$, then we can guarantee that
  $\rho_t$ will weakly converge to $\rho_*$ without the need
  for~\eqref{eq:QP-CLF-constraint}. \oprocend
\end{remark}

\begin{remark}
  \label{rem:discrete-qp}
  \longthmtitle{Discretization and Tractable QP
    Formulation} The infinite-dimensional
  optimization~\eqref{eq:QP-CBF} is not computationally tractable.  A
  discretization approach leading to a practical receding-horizon
  algorithm is the following. Assuming that the support of $\rho_t$,
  $\subscr{u}{nom}$, and $u$ are inside a compact domain $\Omega$, one
  can finitely grid it to obtain piecewise-constant approximations
  $\boldsymbol{\rho}_k$, $\mathbf{U}_k$, and
  $\subscr{\mathbf{U}}{nom,$k$}$ at discrete times $t_k$, $k \ge 0$.
  Using these finite-dimensional vectors, one can
  approximate~\eqref{eq:QP-objective} as the square of a weighted
  two-norm of $\mathbf{U}_k -\subscr{\mathbf{U}}{nom,$k$}$, and
  replace the B-CBF~\eqref{eq:QP-CBF-constraint} and CLF
  constraints~\eqref{eq:QP-CLF-constraint} by suitable (linear)
  discretizations in $\boldsymbol{\rho}_k$ and $\mathbf{U}_k$. In
  general, this results into a finite-dimensional quadratic program
  that can be solved at each time in a receding-horizon fashion: for
  each $\boldsymbol{\rho}_k$, we first find
  $\subscr{\mathbf{U}}{nom,$k$}$, then solve for $\mathbf{U}_k$ from
  the QP, and finally propagate $\boldsymbol{\rho}_k$
  to $\boldsymbol{\rho}_{k+1}$ by means of a suitable discretization
  of~\eqref{eq:macro-LVE}.  We employ this method for
  Example~\ref{implement:1D}.  \oprocend
\end{remark}

\section{Microscopic and Decentralized Algorithms}\label{sec:distributed}

Here, we discuss how to distributedly implement the previous algorithms
at an agent level.

First, we obtain a consistent agent-version of
problem~\eqref{eq:QP-objective}-\eqref{eq:QP-CBF-constraint}, when
using a $\subscr{u}{nom}$ from \eqref{eq:nom-gradient}.  Suppose that
$\{ x_1,\dots,x_N\}$ are $N$ agents' positions sampled independently
from $\rho$. Define $\mathbf{x}_N = (x_1,\dots,x_N)$, and let $r>0$ be
a small constant that determines the communication radius of each
agent. We define the truncated Gaussian kernel
\begin{gather}
  \label{eq:kernel}
  K_r(y) = \frac{1}{C}\exp\left(-\frac{\|y\|^2}{2 r^2}\right)
  \mathbf{1}_{B(0, r)}(y)
\end{gather}
with normalization constant
$C = \int_{B(0,r)} \exp\left(-\frac{\|y\|^2}{2 r^2}\right) \,\diff
x$.
This kernel is used to approximate the density as
\begin{align}
  \label{eq:sampled-density}
  \rhohat(x)
    &= \frac{1}{N} \sum_{j = 1}^N K_r(x - x_j),
\end{align}

Note that this approximation can be computed locally by an agent $i$
provided that it knows $N$, and can communicate with other agents
within a $2r$ distance. More specifically, if each agent $i$ can
communicate in a neighborhood,
$\neighbors_i = \{ j \in \until{N} : \|x_i - x_j\| \le 2r \}$, then
for any $x \in B(x_i, r)$,
\begin{gather}
  \label{eq:rho-hat-i}
  \rhohat(x) = \rhohati(x)
  \triangleq \frac{1}{N} \sum_{j \in \neighbors_i} K_r(x - x_j),
\end{gather}
which follows from the support of the truncated Gaussian kernel
\eqref{eq:kernel}. For this reason, we interpret $\rhohat(x)$ as a
naturally distributed quantity that is available to every agent.  The
following lemma shows uniform bounds and the convergence of the
approximations to the true function values.

\begin{lemma}
  \label{lemma:distr-density}
  Let $x_1, \dots, x_N$ be independent samples from $\rho$. Then, as
  $N \to \infty$ and $r \to 0$, $\rhohat \to \rho$ uniformly a.e.
\end{lemma}
\begin{proof}
  Because the kernel $K_r$ tends to the Dirac delta as $r \to 0$,
  the result follows from \cite[Theorem 3]{VSV:58}.
\end{proof}


The following lemma guarantees the proper convergence of a
sampled-based nominal velocity to a nominal velocity obtained via a
gradient flow.

\begin{lemma}
  \label{lem:sampled}
  Assume that $F:\mathcal{P}(\real^n)\to\real$ is a Fréchet
  differentiable functional such that the gradient of its first
  variation $\nabla \frac{\delta F}{\delta \rho}$ is continuous with
  respect to $\rho$ and that
  $\nabla \frac{\delta F}{\delta \rho}(\rho) \in L^2(\real^n \to
  \real^n)$. In addition, assume that $H$, $\dot{H}$, $V$, and
  $\dot{V}$ are continuous with respect to $\rho$. Define
  \begin{equation*}
    \vnom(\rho) = \nabla \frac{\delta F}{\delta \rho} (\rho).
  \end{equation*}
  Fix a density $\rho\in \mathcal{P}(\real^n)$, and let $u^*$ denote
  the solution of \eqref{eq:QP-CBF} with $\rho_t = \rho$ and nominal
  velocity $\vnom(\rho)$. Let $x_1, \dots, x_N$ be independent
  samples from $\rho$ and let $\bar{u}_{\mathbf{x}_N}$ denote the
  solution of \eqref{eq:QP-CBF} with the approximate density
  $\rho_t = \rhohat$ and nominal velocity $\vnom(\rhohat)$.  Then, as
  $N \to \infty$ and $r \to 0$, $\bar{u}_{\mathbf{x}_N} \to u^*$
  pointwise almost everywhere.
\end{lemma}
\begin{proof}
  By Lemma~\ref{lemma:distr-density} and the continuity assumption,
  $H(\rhohat)$, $\dot{H}(\rhohat)$, $V(\rhohat)$, and
  $\dot{V}(\rhohat)$ converge to their true values $H(\rho)$,
  $\dot{H}(\rho)$, $V(\rho)$, and $\dot{V}(\rho)$, respectively.  By
  assumption, the objective \eqref{eq:QP-objective} is continuous
  with respect to $\rho$. Because the (unique) minimizer of
  \eqref{eq:QP-CBF} depends continuously on the objective and
  constraints, we conclude $\bar{u}_{\mathbf{x}_N} \to u^*$.
\end{proof}

The computation of $\vnom(\rhohat)$ in a distributed manner depends on
properties of $F$.  In particular, if the computation depends only on
$\rhohat$ and $\nabla \rhohat$, it can be computed locally using
\eqref{eq:rho-hat-i}. Another possibility is that the functional
corresponds to $F(\rho)=W_1(\rho,\rho_*)$, for which a distributed
implementation can be found in~\cite{VK-SM:23-auto}. On the other
hand, the constraints \eqref{eq:QP-CBF-constraint} and
\eqref{eq:QP-CLF-constraint} may pose a bigger challenge. The problem
can be further decoupled and solved in a distributed manner if there
is an algorithm that allows each agent to find a local solution by
communicating with other agents.

The following result identifies a condition under which the problem
structure lends itself to a distributed implementation. To do so, we
make use of a Voronoi partition, $\{\Vor_i\}_{i=1}^N$ generated by the positions
of agents $\mathbf{x}_N$; see~\cite{FB-JC-SM:09}.

\begin{theorem}\label{le:distributed-B-CBF}
  Suppose that all the
  conditions of Lemma~\ref{lem:sampled} hold.  Assume that
  $\supp(\rho_t), \supp(\subscr{u}{nom}) \subseteq \Omega \subseteq
  \real^n$ is a compact domain, that $\rho_t$ is absolutely
  continuous, and $\alpha$ is linear. Suppose that the functional $H$
  (resp.~$V$) is given by
  $H(\rho) = \int_{\Omega} h(x,\rho) \,\diff x$,
  (resp.~$V(\rho) = \int_\Omega g(x,\rho) \, \diff x$) where
  $h(x,\rho)$ (resp.~$g(x,\rho)$) is integrable, has integrable
  partial derivatives, and is continuous.  Define
  $H_i(\rho) = \int_{B(x_i,r)\cap \Vor_i} h(x,\rho)\,\diff x$,
  (resp.~$V_i(\rho) = \int_{B(x_i,r)\cap \Vor_i} g(x,\rho)\,\diff x$)
  $i \in \until{N}$, and consider the solution
  $\subscr{\bar{u}}{{distr,$\mathbf{x}_N$}}$ to
  \begin{subequations}\label{eq:distr-bcbf}
    \begin{align}
      \arg\min_{\bar{u}\in \U}& \ \sum_{i=1}^N
                            \int_{B(x_i,r)\cap \Vor_i} \hspace*{-1.2cm}
                            \| \bar{u}(x)-\vnom(\rhohati)(x)\|^2 \diff x
                            \label{eq:distr-bcbf-objective}\\[-0.2em]
      \text{s.t.}
      \sum_{i=1}^N &\dot{H}_i(\rhohati(x)) \ge - \sum_{i=1}^N \alpha_2(H_i( \rhohati(x))) \label{eq:distr-bcbf-constraint},
      \\
      \text{}   \sum_{i=1}^N &\dot{V}_i(\rhohati(x)) \le - \sum_{i=1}^N \alpha_2( V_i( \rhohati(x)) ). \label{eq:dist-QP-CLF-constraint}
    \end{align}
  \end{subequations}

  Then, it holds that
  $\subscr{\bar{u}}{{distr,$\mathbf{x}_N$}} \to u^*$, as $N\to \infty$
  and $h \to 0$, where $u^*(t,x)$ is a solution
  of~\eqref{eq:QP-objective}-\eqref{eq:QP-CLF-constraint}.
\end{theorem}

\begin{proof}
  The cost in \eqref{eq:distr-bcbf-objective} can be rewritten as
\begin{equation*}
  \int_{\real^n} \sum_{i=1}^N
  \mathbf{1}_{B(x_i,r)\cap \Vor_i} \| \bar{u}_t(x)-\vnom(\rho)(x) \|^2 \diff x.
\end{equation*}
By assumption, the integrand is bounded by an integrable
function. Then, by the Dominated Convergence Theorem (DCT), its limit
as $N \to \infty$, and $h \to 0$, converges to
$ \int_{\Omega} \| \bar{u}_t(x)-\vnom(\rho)(x)\|^2 \diff
x$.  We can rewrite \eqref{eq:distr-bcbf-constraint} as
\begin{equation*}
  \sum_i H_i(\rhohati(x)) = \int_{\real^n}
  \sum_{i=1}^N \mathbf{1}_{B(x_i,r)\cap \Vor_i}  h(x,\rhohati(x))\diff x,
\end{equation*}

Lemma~\ref{lemma:distr-density}, which shows convergence of
$\rhohati(x)$ and $\nabla\rhohati(x)$ over $B(x_i,r)$, the
integrability and continuity of $h$ and its derivatives, together with
the DCT, ensure
$\sum_{i=1}^N \alpha(H_i(\rhohati))=\alpha(\sum_{i=1}^N
H_i(\rhohati))\to \alpha(H(\rho))$ as $N\to\infty$ and $h\to 0$, where
we used $\alpha$ is linear.  We can use the same argument to show
$\sum_{i=1}^N \dot{H}_i( \rhohati(x))\to \dot{H}( \rho)$ as
$N\to\infty$, $h\to 0$. An analogous argument can be made about
constraint~\eqref{eq:dist-QP-CLF-constraint}.  The result follows from
the continuity of the minimizer with respect to the objective and
constraints.
\end{proof}

   %
  %
   %

\begin{remark}[Spatially-Dependent Functionals]\label{remark:distr-spatial}
  A direct approach can be used in the case of spatially dependent
  functionals such as $H(\rho) = \subscr{\rho}{max} - \rho$.  In this case,
  restricting the B-CBF constraint to a Voronoi partition as in Lemma~\ref{le:distributed-B-CBF},  decouples
 the inequality in each Voronoi region $\Vor_i$ as follows
   \begin{equation}\label{eq:voronoi-spatial}
     \nabla \cdot \rhohati(x) u(x)
     \ge \alpha_{x_i}(\subscr{\rho}{max}(x) -
     \rhohati(x)),
   \end{equation}
   $\forall x \in \Vor_i\cap B(x_i,r)$. These constraints will
   converge to the original constraint from Lemma~\ref{lemma:distr-density} as
   $N \rightarrow \infty$,~$h\rightarrow 0$. \oprocend



  \end{remark}
%




\begin{remark}[Distributed Implementations]
  \label{remark:distr-impl}
  To solve~\eqref{eq:distr-bcbf} with local information, each agent
  $i$ must compute first its local constraints and objective over the
  regions $\Vor_i \cap B(x_i,r)$. As discussed in~\cite{FB-JC-SM:09}, the
  calculation of $\Vor_i \cap B(x_i,r)$ can be done with knowledge of the
  positions of other agents within $2h$
  distance. 
To compute the derivatives, we apply Lemma~\ref{lem:General-CBF},
 extra terms arise from the time-dependent boundary, although these would vanish
if an unrestricted Voronoi partition was used, as noted in~\cite{FB-JC-SM:09}.
  The decision variable for each agent is then $\overline{u}_t(x)$
  over $\Vor_i \cup B(x_i,r)$. Under general integral constraints, the
  optimization is coupled and possibly nonlinear, as each decision
  variable will appear when computing the derivative of the
  functionals
  in~\eqref{eq:distr-bcbf-constraint}-\eqref{eq:dist-QP-CLF-constraint}. Thus,
  in order to implement the optimization, a quasi-static approximation
  can be applied by freezing agents' positions and the corresponding
  density.  Agents can then apply primal-dual algorithms for
  distributed continuous-time optimization such as
  in~\cite{AC-JC:15-tcns} (or as in~\cite{MZ-SM:09c}, after
  discretization) introducing Lagrange multipliers and consensus
  variables.  
  A simple space discretization consists of assuming constant values
  $\overline{u}_t(x) \equiv \overline{u}_t^i$, for
  $x \in \Vor_i \cap B(x_i,r)$, which reduces the integrals to sums.
  Once the velocity is computed, the positions of agents are propagated
  by using the simple integrator dynamics that governs their motion.
  %
  %
\oprocend
\end{remark}

\section{Numerical Examples} \label{sec:numerical}
In this section, we apply our method to the following two cases.
We refer to more examples with other B-CBFs in our extended
version~\cite{XG-GP-SB-SM:25}.
%
\subsection{Example: 1D Transport with density constraints}
\label{implement:1D}

We consider a 1D transport problem from an initial Gaussian density
$\rho_0 \sim \mathcal{N}(0, 1)$ to a target uniform density
$\rho_* \sim \mathcal{U}[10,14]$ on the time horizon $t \in [0,1]$.
%
%
We impose a pointwise density cap
(cf. Lemma~\ref{le:density-threshold}), $\rho(t,x) \le 0.2$, for all
$x$ in the interval $S = [3,7]$. We apply the B-CBF
$H : \mathcal{P}(\real) \to L^1(S)$ defined by
$H(\rho) = 0.2 - \rho\vert_S$, where $\rho\vert_S$ denotes the
restriction of $\rho$ to the domain $S$.
%
%

  We employ a time discretization with $\Delta t = 0.001$, and
  $t_k = k \Delta t$. At each $t_k$, we firstly compute the
  time-varying optimal transport map following Lemma~\ref{lemma:1D-OT}
  and choose $\gamma(t_k) = 1/(1-t_k)$, then solve the optimization
  problem~\eqref{eq:QP-CBF} by gridding the line. The solved velocity
  field $\mathbb{U}_k$ is applied to $\mathbb{\rho}_k$ for updating
  the density with implicit upwind method.
\begin{figure}[htbp]
  \centering
  \includegraphics[width = 0.48\textwidth]{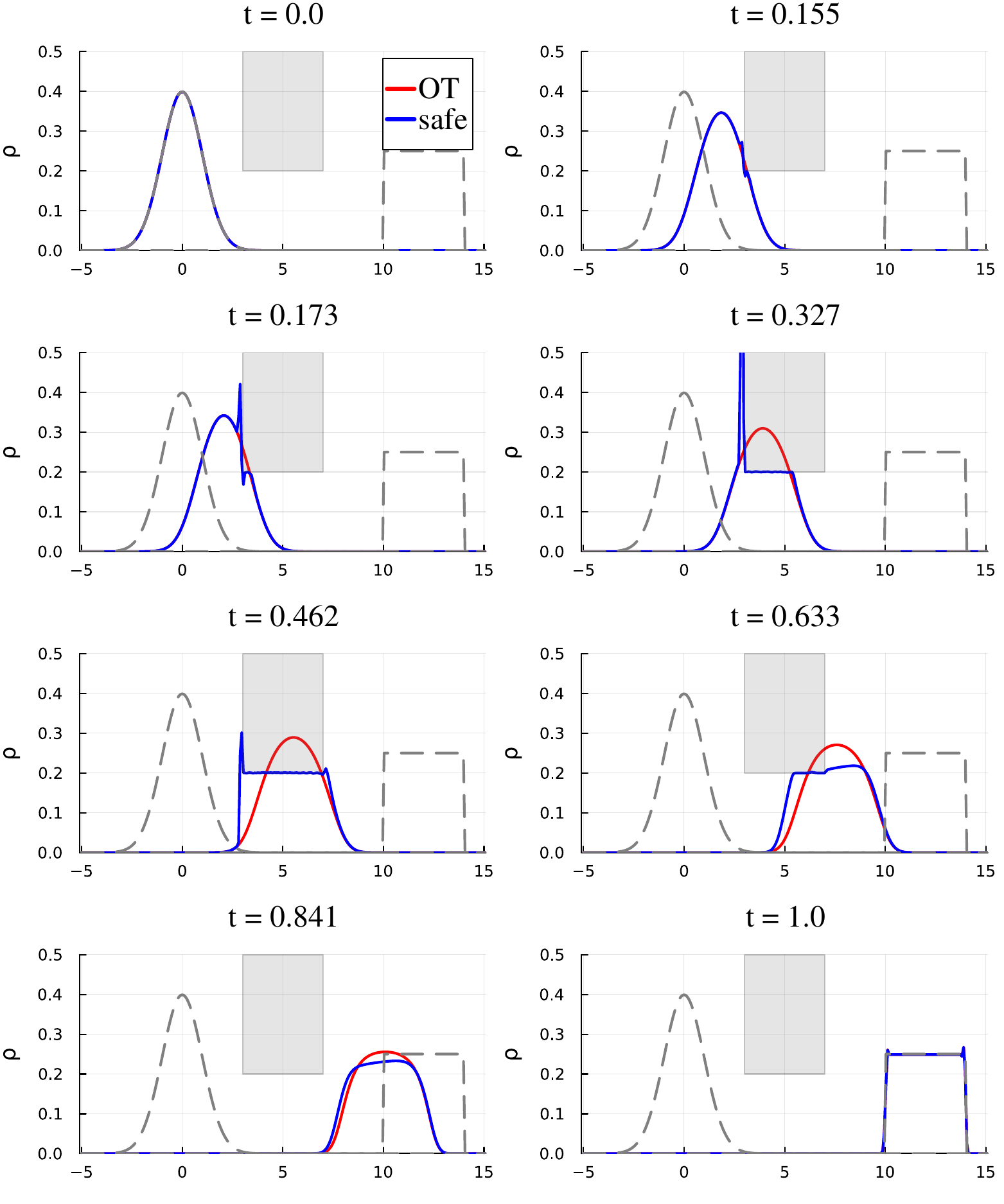}
  \vspace{-8mm}
  \caption{The gray dashed curves represent the initial and terminal density
  $\rho_0, \rho_*$, and the unsafe area is shaded.
  The red and blue curves show the evolution of unconstrained OT and safe densities
  after CBF projection respectively.}
  \label{fig:1D-OT}
\end{figure}
As shown in Figure \ref{fig:1D-OT}, the safety filter activates as the
density approaches the constrained region. The resulting velocity
field slows and spreads the density to satisfy the constraint
$\rho(t,x) \le 0.2$, before re-converging to successfully complete the
transport to the target distribution.

\subsection{2D Obstacle Avoidance}
\label{implementation:2D-obstacle}

The obstacle avoidance formulation in Lemma~\ref{lemma:Obstacle} can 
be applied to a 2D implementation of Optimal Transport. To do that,
we consider an initial and target distribution as follows:
\begin{equation*}
  \rho_0 \sim \mathcal{N}\Big(\begin{bmatrix}1\\1\end{bmatrix}, \tfrac{1}{20} I_2\Big), \quad
  \rho_T \sim \mathcal{N}\Big(\begin{bmatrix}7\\7\end{bmatrix}, \tfrac{1}{20} I_2\Big),
\end{equation*}
and an obstacle $O ^b\subset \real^2$. We then use the B-CBF
\begin{equation*}
   H(\rho) = \epsilon - \int_{O^b} \rho(x)\,dx ,
\end{equation*}
where $\epsilon > 0$ represents the maximum allowable density inside the
obstacle. 

To obtain the nominal velocity we apply a Sinkhorn Algorithm to compute the optimal coupling, and then 
a barycentric approximation to determine the transport map,
as explained in~\cite{GP-MC:17}.
After that we solve~\eqref{eq:QP-CBF} with the same
 discretization as in the previous example in 2D.
We apply an implicit upwind method to propagate the density.
Figure~\ref{fig:2D-OT-obstacle} shows how once the
barrier functional is enforced, the velocity field diverts the flow
around the obstacle, preserving the safety constraint while still
steering the density toward the target distribution.

\begin{figure}[htbp]
  \centering
  \includegraphics[width = 0.48\textwidth]{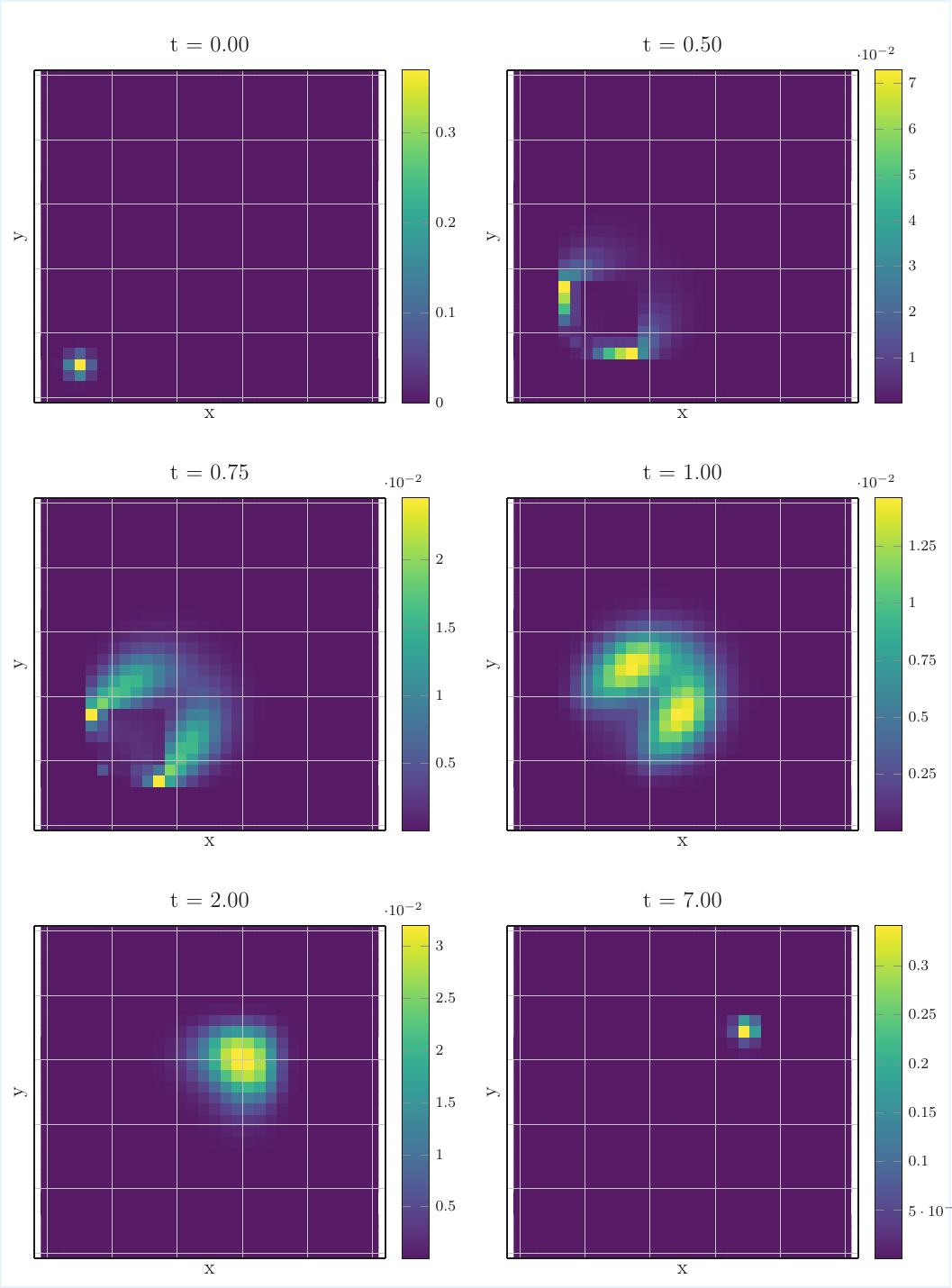}
  \vspace{-6mm}
  \caption{Heatmap depicting snapshots of the density evolution under obstacle avoidance in 2D.
The plots depict the transport from the initial to the target density
while respecting the square obstacle constraint. As time progresses, the
density flow diverts around the obstacle, successfully avoiding unsafe
regions and converging toward the target distribution.}
  \label{fig:2D-OT-obstacle}
\end{figure}
\subsection{2D Distributed OT with Conflict Avoidance and Swarm Cohesion}
\label{implementation:2DOT}
To compute the optimal transport velocity field in a distributed
manner, we follow the implementation in \cite{VK-SM:23-auto} to find a
nominal velocity.  We then leverage the formulation in
Lemma~\ref{lemma:CG}, together with the
Remarks~\ref{remark:distr-spatial},~\ref{remark:distr-impl} to
implement the B-CBF ensuring collision avoidance and maintaining swarm
cohesion.  After using a piecewise constant approximation of $u_t$ and
$\subscr{\widehat{u}}{nom,$\mathbf{x}_N$}$ over each agent's region,
we can completely decouple problem~\eqref{eq:distr-bcbf} and obtain a
local optimization of the form
\begin{align*}
  v_i^*(t) = \argmin_{u_t} &\;\;
     \|u_t -\subscr{\widehat{u}}{nom,$\mathbf{x}_N$}(t,x_i)\|^2_2 \\
  \text{s.t.}\;\;
       u_t&\cdot\nabla\rhohati(x_i)
        \geq \alpha_1\big(\rhohati(x_i) - \subscr{\rho}{max}(x_i)\big), \\
       u_t&\cdot\nabla\rhohati(x_i)
        \leq \alpha_2\big(\rhohati(x_i) - \subscr{\rho}{min}(x_i)\big).
\end{align*}

For the implementation we use a constant minimum and maximum densities
$\subscr{\epsilon}{min},\subscr{\epsilon}{max}$.  After adding the
safety filter for each agent, we obtain the results depicted in Figure
\ref{fig:2D-OT}. As anticipated, the agents spread apart in order to
reduce the local density.  This might create a mismatch with the
target distribution, since in some regions the target density is too
high for the agents to accumulate.  At the same time, there are no
isolated agents: even when agents remain far from the target, they
still appear in clusters.
\begin{figure}[htbp]
  \centering
  \includegraphics[width = 0.48\textwidth]{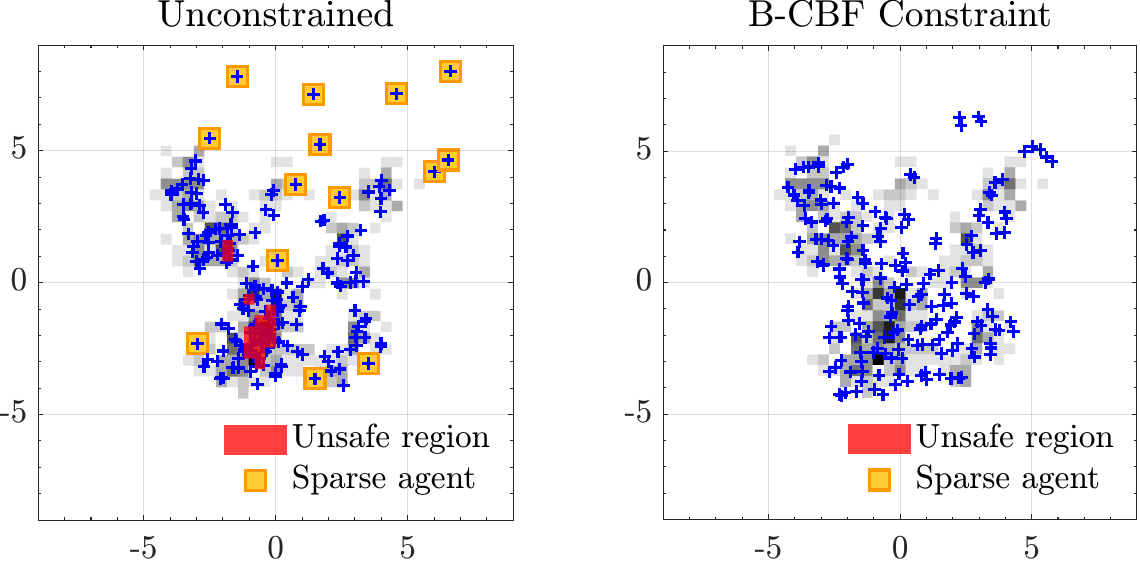}
  \label{fig:density_threshold}
  \vspace{-6mm}
  \caption{Final state of agent distribution for the unconstrained
    (left) and CBF constrained (right) distributed optimal
    transport. The target density is shown in grayscale, and the
    maximum density constraint violation is highlighted in red. The
    constraint corresponds to a maximum density of
    $\subscr{\epsilon}{max}=0.045$, and a minimum density of
    $\subscr{\epsilon}{min}=0.01$ . }
  \label{fig:2D-OT}
\end{figure}
\subsection{Distributed Entropy Lower Bound}
In this scenario, we use the nominal controller given by the distributed optimal transport algorithm presented in the previous case and apply an integral-type constraint, corresponding to a lower bound on the entropy of the swarm:
\begin{equation*}
    H(\rho)=\epsilon-\mathcal{H}(\rho),
\end{equation*}
where $\mathcal{H}(\rho)=-\int_\Omega\rho log\rho\diff x$ is the entropy
functional, and the safe region is the set $\{H(\rho)\geq0\}$, which is
equivalent to a lower bound on the entropy $\{\mathcal{H}(\rho)\geq \epsilon\}$.
We then make use of Theorem\eqref{le:distributed-B-CBF}, to implement the B-CBF
in a distributed way. We assume that the value of the functional and the nominal
controller is constant inside the cells of the partition
$B(x_i,h)\cap\mathcal{V}_i$, applying this case to problem~\eqref{eq:distr-bcbf}
leads to the following coupled optimization problem:
\begin{subequations}
\begin{align*}
  \arg\min_{\bar{u}\in \mathcal{U}} &\ \sum_{i=1}^N
    \big\|\bar{u}_i-\widehat{u}_{\mathrm{nom},\mathbf{x}_N}(x_i)\big\|^2 \,w_i
  \\[-0.2em]
  \text{s.t.}\qquad
  &\sum_{i=1}^N \nabla \rhohati(x_i)\,\bar{u}_i \,w_i
  \label{eq:constraint}
  \\
  &\qquad\ge -\alpha\Bigg(\epsilon + \sum_{i=1}^N \,\rhohati(x_i)\,\log\!\big(\rhohati(x_i)\big)\,w_i\Bigg),
\end{align*}
\end{subequations}
where $w_i$ is the size of the cell. Notice how this optimization has the
structure of a QP with a global coupling constraint. When implementing a
distributed algorithm to solve the optimization problem, we need to guarantee
anytime feasibility, ensuring that whenever we stop the algorithm, the obtained
approximated solution still guarantees the safety of the system. To do that we
solve the next local problem for each agent:
\begin{subequations}
\begin{align*}
  v_i^* =&\ \arg\min_{u}  \big\|u-\widehat{u}_{\mathrm{nom},\mathbf{x}_N}(x_i)\big\|^2 
  \\[-0.2em]
  \text{s.t.}\hspace{0.1cm}
  & \alpha\,\rhohati(x_i)\,\log\!\big(\rhohati(x_i)\big)\,w_i
    - \nabla \rhohati(x_i)\,w_i\,u
  \\
  & \qquad\le \sum_{j\in\mathcal{N}_i}(y_i-y_j) + \alpha\frac{\epsilon}{N},
\end{align*}
\end{subequations}

where  $y_i\in\real, i\in\{1,\dots,N\}$,  are slack variables. Crucially, the
aggregate solution $\bar{v}^*=(v_1^*,\dots,v_N^*)$ satisfies the global
constraint for any set of slack variables, because $\sum_{i}\sum_{j\in
\mathcal{N}_i} (y_i - y_j) = 0$ over an undirected graph. Consequently, safety
is maintained throughout the optimization process. Since the local problem is a
QP, it can be solved efficiently. To ensure the local solution vector
$\bar{v}^*$ converges to the global optimality, the slack variables are updated
via the dynamic consensus rule:
\begin{equation*}
    \dot{y}_i=-k\sum_{j\in\mathcal{N}_i}(\lambda_i-\lambda_j),
\end{equation*}
where $\lambda_i$ is the Lagrange multiplier associated with agent $i$'s local
constraint. The convergence of this algorithm to the global minimizer is
established in \cite{XT-CL-KHJ-DVD:25}. Figure~\ref{fig:entropy_cap} shows the
comparison between the unconstrained and constrained evolution of the entropy in
the optimal transport for the swarm. As one can check, in the constrained case,
the entropy converges to the safe region, and remains in the safe set throughout
the rest of the simulation.

\begin{figure}[htbp]
  \centering
  \includegraphics[width = 0.4\textwidth]{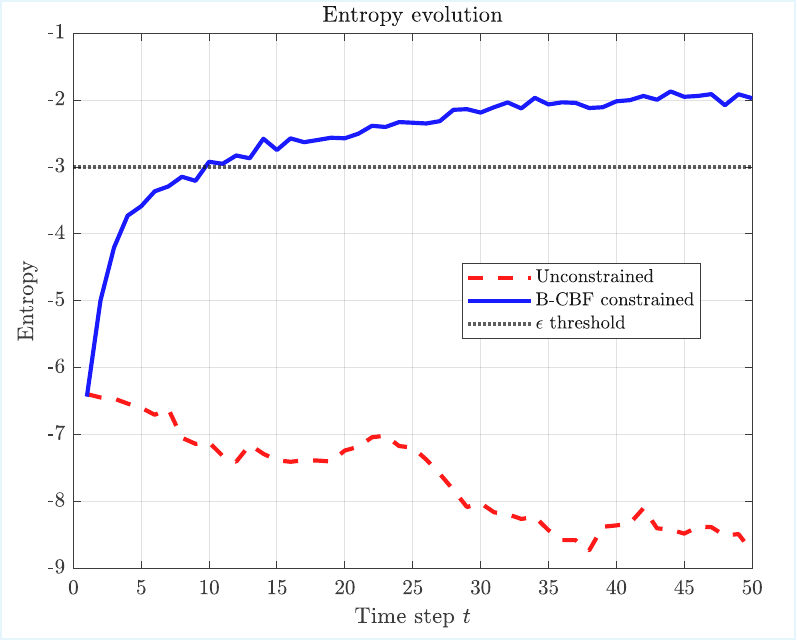}
  \vspace{0.05cm}
  \caption{Evolution of the entropy for the unconstrained and B-CBF constrained case, with an entropy threshold of $\epsilon=3$.}
  \label{fig:entropy_cap}
\end{figure}
\section{Conclusion and Future work}\label{sec:conclude}
In this work, we introduced Banach Control Barrier Functions
(B-CBFs) generalizing classical and mean-field CBF to  ensure safety in
large-scale swarm systems. We demonstrated how B-CBFs integrate with
gradient-flow controllers, particularly those derived from optimal
transport, to establish formal guarantees on both safety and
convergence. Furthermore, we developed a consistent, distributed
algorithm that allows individual agents to implement the macroscopic
control law, bridging the gap between microscopic execution and
macroscopic safety guarantees. The effectiveness of this scalable and
robust approach was validated through numerical simulations. Future
research directions include applications to stochastic systems, more
general constraints, and bilevel distributed optimization problems.

\bibliographystyle{IEEEtran}
\bibliography{alias,SMD-add,SM,JC}

\end{document}